\documentclass[12pt, reqno]{amsart}
\usepackage{amsmath, amsthm, amscd, amsfonts, mathrsfs, amssymb, graphicx, xcolor, float}
\usepackage[bookmarksnumbered, colorlinks, plainpages]{hyperref}
\input{mathrsfs.sty}
\hypersetup{colorlinks=true,linkcolor=red, anchorcolor=green, citecolor=cyan, urlcolor=red, filecolor=magenta, pdftoolbar=true}

\usepackage{cancel}

\newtheorem{theorem}{Theorem}[section]
\newtheorem{lemma}[theorem]{Lemma}

\newtheorem{corollary}[theorem]{Corollary}
\theoremstyle{definition}
\newtheorem{definition}[theorem]{Definition}
\newtheorem{example}[theorem]{Example}

\theoremstyle{remark}
\newtheorem{remark}[theorem]{Remark}
\numberwithin{equation}{section}

\usepackage[T1]{fontenc}

\let\<\langle
\def\ch{\raise 0.5ex \hbox{$\chi$}}

\let\\\cr

\let\epsilon\varepsilon

\begin{document}

\title[Convergence of Random Products of Projections]{Convergence of Random Products of Projections Under Infinite-Periodic Selections}
	
\author[R. Eskandari]{Rasoul Eskandari}
\address{Department of Mathematics Education, Farhangian University, P.O. Box 14665-889, Tehran, Iran.}
\email{Rasoul.eskandari@cfu.ac.ir, eskandarirasoul@yahoo.com}	
	
\author[M. S. Moslehian]{Mohammad Sal Moslehian$^*$}
\address{Department of Pure Mathematics, Faculty of Mathematical Sciences, Ferdowsi University of Mashhad, P. O. Box 1159, Mashhad 91775, Iran}
\email{moslehian@um.ac.ir; msmoslehian@gmail.com}

\dedicatory{Dedicated to Professor Mehdi Radjabalipour with respect and affection}	
	
\thanks{$^*$Corresponding author}
	
	\begin{abstract}		
	Let $\{P_j\}$ be a sequence of projections onto the closed subspaces $\mathcal{M}_j$ of a Hilbert space $\mathscr{H}$. 
	Consider an infinite sequence $P_{i_1}, P_{i_2}, \ldots$ with each $P_{i_n} \in \{P_1, P_2, \ldots\}$, possibly repeating in some order or randomly. 
	The question is: Under what conditions does the sequence $\{P_{i_n} \cdots P_{i_2} P_{i_1} x\}_{n=1}^{\infty}$ converge strongly or weakly to $Px$ for every $x \in \mathscr{H}$, where $P$ is the projection onto the intersection $\mathcal{M} = \bigcap_{i=1}^{\infty} \mathcal{M}_i$?
	
	In this paper, we present some results concerning random products of countably infinitely many projections $\{P_j\}_{j=1}^{\infty}$ that incorporates the notion of an infinite-periodic function. More precisely, we introduce a new class of functions $\sigma \colon \mathbb{N} \to \mathbb{N}$, called infinite-periodic functions, and rigorously show that the sequence $\{T_n x\}$ defined by 
	\[
	T_1 := P_{\sigma(1)}\quad \mbox{and} \quad T_n := P_{\sigma(n)} T_{n-1} \quad \text{for all } n \geq 2,
	\]
	for $x \in \mathscr{H}$, converges weakly to $Px$, where $P$ is the projection onto $\bigcap_{j=1}^{\infty} \mathcal{R}(P_j)$. We also provide some technical examples to illustrate our results. 
\end{abstract}

\maketitle

\section{Introduction}
The $C^*$-algebra of all bounded linear operators on a complex Hilbert space $(\mathscr H, \langle \cdot, \cdot \rangle)$ is denoted by $\mathbb{B}(\mathscr H)$. For an operator $T$, let $\mathcal{R}(T)$ and $\mathcal{N}(T)$ denote its range and kernel, respectively. By a projection in $\mathbb{B}(\mathscr H)$, we mean an orthogonal projection.

Let $P_1,\ldots,P_r$ be projections onto closed subspaces $\mathscr M_1,\ldots,\mathscr M_r$, respectively, and let $P$ be the projection onto $\bigcap_{i=1}^r \mathscr M_i$. An interesting question can be stated as follows: Under what conditions does the randomly generated sequence
\[
\{P_{i_n} \cdots P_{i_2} P_{i_1} x\},
\]
where each $P_{i_n}$ is chosen randomly from the set $\{P_1,\ldots,P_r\}$, converge strongly or weakly to $Px$ in $\mathscr H$?

This problem is inherently complicated, particularly in the case $r = \infty$; see \cite{DYE, DYE2, SAK}.

One application of this problem lies in finding and approximating solutions of systems of operator equations. Let $A_i \in \mathbb{B}(\mathscr H)$ for all $1 \leq i \leq n$ and consider the system
\begin{equation}\label{system}
	A_i x = 0 \qquad (1 \leq i \leq n).
\end{equation}
Define $\mathscr M_i = \{x \in \mathscr H : A_i x = 0\}$, and let $P_i$ be the projection onto $\mathscr M_i$. Then, if the limit of $\{P_{i_n} \cdots P_{i_2} P_{i_1} x\}$ exists, it is a solution of \eqref{system}; see \cite{DEU2}.

Strong convergence for the case of two projections (alternating between two closed subspaces) was first proved by von Neumann in \cite{VON}. Halperin \cite{N} later extended von Neumann's result to finitely many projections (more than two subspaces). These results are summarized in the following theorem.

\begin{theorem}(von Neumann Alternating Projections Theorem) Let $P_1, \ldots, P_r\in \mathbb{B}(\mathscr{H})$ be projections. For each $x \in \mathscr H$, it holds that
	\[
	\lim_{n \to \infty} \| (P_r P_{r-1} \cdots P_1)^n x - P x \| = 0,
	\]
	where $P$ is the projection onto $\bigcap_{i=1}^r \mathcal{R}(P_i)$.
\end{theorem}

Readers seeking simple geometric proofs may consult the papers of Kopeck\'{a} and Reich \cite{KR1, KR2}. This result has numerous generalizations and applications, which can be explored further in sources such as \cite{BS} and a series of articles by Reich \cite{PRZ}-\cite{RZ2}. In this theorem, the projections appear in a fixed cyclic order. When projections are applied in an arbitrary order, Amemiya and Ando \cite{AA} demonstrated that if only finitely many projections are involved, the sequence $\{P_{i_n} \cdots P_{i_2} P_{i_1} x\}_{n=1}^\infty$ converges weakly, where each $P_{i_n}$ belongs to $\{P_1, \ldots, P_r\}$. They conjectured that the conclusion remains true if ``weakly" is strengthened to ``strongly''. This conjecture was disproven by Paszkiewicz \cite{PAS} (with $r=5$) and independently by Kopeck\'{a} and V. M\"{u}ller \cite{KM} (with $r=3$). As Kopeck\'{a} and Paszkiewicz \cite{KP} further illustrated, projections can exhibit highly irregular behavior. In particular, in an infinite-dimensional Hilbert space $\mathscr{H}$, one can find three projections $P_1, P_2, P_3$ such that for every nonzero vector $x\in\mathscr{H}$, there exists a sequence of indices $i_1, i_2, \ldots \in \{1,2,3\}$ for which $\{P_{i_n}\cdots P_{i_2}P_{i_1}x\}$ fails to converge strongly.

In \cite{BOR1, BOR2}, weak  convergence of consecutive projections onto infinite families of convex sets is also investigated, although the considered conditions are of a metric rather than combinatorial nature. The combinatorial machinery developed in \cite{BOR3} may be useful for further investigation of random products of projections. They investigate the structure of the set of all partial weak limits of the sequence generated by alternating projections onto closed convex sets in a Hilbert space.

Sakai \cite{SAK}  extended Halperin's result to the setting of quasi-periodic functions with finitely many projections. He then raised the significant question of whether his findings extend to a class broader than quasi-periodic functions, or even to the case of countably many projections. This question was taken up in \cite{MDD}, where the authors considered a more general class of quasi-periodic sequences \cite{THI, THI2}. 
There has been very little investigation into the case involving countably many projections (see \cite{EM, EBM}), which underscores the significance of this work. 

In \cite{EM}, pseudo-periodic functions were defined as a generalization of quasi-periodic functions, thereby extending Sakai's results. The present work enriches this line of research by introducing the notion of infinite-periodic functions and establishing the weak convergence of sequences given by countable infinite products of projections associated with such functions. We also provide some technical examples to illustrate our results.

\section{Convergence of products of an infinite number of projections}

Let $r$ be a positive integer. Let $\sigma:\mathbb{N}\to \{1,2,\cdots,r\}$ be a function. Following \cite[p. 2]{SAK}, we say that $\sigma$ is a \emph{quasi-periodic function} if there exists $m\in \mathbb{N}$ such that 
\[
\left\{\sigma(k+1),\sigma(k+2),\cdots,\sigma(k+m)\right\}=\left\{1,2,\cdots,r\right\}
\]
for all $k\in \mathbb{N}$.

Now we recall the definition of a pseudo-periodic function, introduced in \cite[Definition 2.3]{EM}, which generalizes the notion of quasi-periodic function.

\begin{definition} \label{pse}
	Let $J$ be either a finite or an infinite subset of $\mathbb{N}$. Suppose that $\sigma:\mathbb{N}\to J$ is a surjective function. 
	For each fixed $j$, let $\sigma^{-1}(j)$ be an infinite set and let $\{l_{n}\}_{n=1}^\infty$ be the strictly increasing sequence (depending on $j$) of all natural numbers such that $\sigma(l_{n})=j$, and set $l_{0}=0$. 
	Let $\Gamma_F$ denote the set of all $j$ such that 
	\begin{align}\label{sup}
		I(\sigma,j) = \sup_n \left(l_{n}-l_{n-1}\right) < \infty.
	\end{align}
	If we set $\Gamma_\infty := \sigma(\mathbb{N}) \setminus \Gamma_F$, then $\Gamma_F$ and $\Gamma_\infty$ are disjoint subsets of $\mathbb{N}$ with $\Gamma_F \cup \Gamma_\infty = \sigma(\mathbb{N})$. It is obvious that if $J$ is a finite subset of $\mathbb{N}$ and $\Gamma_F = \sigma(\mathbb{N})$, then $\sigma$ is a quasi-periodic function.
	
	Let $\{k_n\}$ denote the complement (arranged in increasing order) of the union of all sequences $\{l_n\}$  as $j\in \Gamma_{F}$ varies. Clearly, $\sigma(k_n)\in \Gamma_\infty$.
	
	A function $\sigma:\mathbb{N}\to J$ is called a \emph{pseudo-periodic function} if 
	\[
	\Gamma_F = \{1, 2, \ldots, r\}
	\]
	for some positive integer $r$, and the sequence $\{k_n-k_{n-1}\}$ is increasing. 
	In this case, there exists an integer $m \geq r$ such that for every $k \geq 0$,
	\[
	\{1, 2, \ldots, r\} \subseteq \{\sigma(k+1), \sigma(k+2), \ldots, \sigma(k+m)\}.
	\]
	The smallest such $m$ is called the \emph{period} of $\sigma$.
\end{definition}

The terms ``quasi-periodic'' and ``pseudo-periodic'' have firmly established definitions for functions on the real line, yet when applied to sequences, they appear to be used in ways that diverge from their traditional senses.

\begin{definition} \label{gse} 
	Let $J$ be an infinite subset of $\mathbb{N}$ and let $\sigma : \mathbb{N} \to J$ be a surjective function. Let $\sigma^{-1}(j)$ be an infinite set, for all $j\in J$. 
	Let $\{l_{n}\}$, $I(\sigma,j)$, $\Gamma_F$, $\Gamma_\infty$, and $\{k_n\}$ be as in Definition \ref{pse}. 
	
	The function $\sigma$ is called an \emph{infinite-periodic function} if 
	$\Gamma_{\infty}$ is a finite set. This condition guarantees that if $\sigma$ is infinite-periodic, then $\Gamma_F$ is an infinite set.
\end{definition}

\begin{example}\label{exam}
	Let $\{P_i : i \geq 1\}$ be a sequence of projections on a Hilbert space $\mathscr{H}$.  Let $S=\{2^j+1: j\geq 0\}$. Define $\sigma : \mathbb{N} \to \mathbb{N}$ by
	\begin{equation}\label{eqexampleg}
		\sigma(n):= \begin{cases}
			i+3 & n = 2^{i-1}(2k-1),\; n\not\in S, \; i,k \in \mathbb{N}\\
			s+1 & n=2^{3k+s}+1,\; 0\leq s\leq 2,\; k\in \mathbb{N}\cup\{0\}
		\end{cases}.
	\end{equation}
	For instance, for $i = 1$, using \eqref{eqexampleg}, the set of all $l_n$ satisfying $\sigma(l_n) = 1$ is 
	\[
	\{2^0+1, 2^3+1, 2^6+1, 2^{9}+1, \ldots\}.
	\]
	For $i = 2$, the set of all $l_n$ satisfying $\sigma(l_n) = 2$ is 
	\[
	\{2^1+1, 2^4+1, 2^7+1, 2^{10}+1, \ldots\}.
	\]
	For $i = 3$, the set of all $l_n$ satisfying $\sigma(l_n) = 3$ is 
	\[
	\{2^2+1, 2^5+1, 2^8+1, 2^{11}+1, \ldots\}.
	\]
	For $i = 4$, the set of all $l_n$ satisfying $\sigma(l_n) = 4$ is 
	\[
	\{1, 7, 11, 13,15, \ldots\}.
	\]
	For $i = 5$, the set of all $l_n$ satisfying $\sigma(l_n) = 5$ is 
	\[
	\{6, 10, 14, 18,22,26,30,34, \ldots\}.
	\]
	From \eqref{eqexampleg} we obtain, for each $i \geq 4$,
	\[
	I(\sigma,4) = 6,\; I(\sigma,5) = 4,\; I(\sigma,6) = 8,\; \ldots,\; I(\sigma,i) = 2^i.
	\]
	Hence $\sigma$ is an infinite-periodic function. In this case,
	\[
	\Gamma_F = \{ i : i \geq 4 \}\quad \mbox{and} \quad \Gamma_\infty = \{1,2,3\} .
	\]
	
	Now, let $\sigma$ be an infinite-periodic function. Fix $x \in \mathscr H$ and define
	\begin{equation}\label{eq define Tn}
		T_1 x = P_{\sigma(1)} x\quad \mbox{and} \quad T_n x = P_{\sigma(n)} T_{n-1} x \quad (n \geq 2).
	\end{equation}
	We claim that $\{ T_n x \}$ converges weakly to $Px$, where $P$ is the projection onto $\bigcap_{i=1}^\infty \mathcal{R}(P_i)$. To prove this we need the following lemmas.
\end{example}

\begin{lemma}\label{sakai}\cite[Lemma 2.5]{EM}
	Let $J$ be either an infinite or a finite subset of $\mathbb{N}$.
	Let $P_n \in \mathbb{B}(\mathscr{H})$ be a projection onto a closed subspace $\mathscr{M}_n$ of $\mathscr{H}$ for each $n \in J$. 
	Let $\sigma:\mathbb{N} \to J$ be an arbitrary function and define $T_1 := P_{\sigma(1)}$ and $T_n := P_{\sigma(n)} T_{n-1}$ for $n \geq 2$. Then
	\[
	\lim_{n \to \infty} \|T_{n-k} x - T_n x\| = 0
	\]
	for each $k \geq 1$.
\end{lemma}

\begin{lemma}\label{lemma weak}\cite[Lemma 2.8]{EM}
	Let $\{x_n\}$ be a bounded sequence in a Hilbert space $\mathscr{H}$ and let $x_0 \in \mathscr{H}$. 
	If every weakly convergent subsequence of $\{x_n\}$ converges weakly to $x_0$, then the whole sequence $\{x_n\}$ converges weakly to $x_0$, that is,
	\[
	w\!-\!\lim_{n\to\infty} x_n = x_0.
	\]
\end{lemma}

\begin{lemma}\label{lemma p}
	Let $P_n \in \mathbb{B}(\mathscr{H})$ be projections onto closed subspaces $\mathscr{M}_n$ of $\mathscr{H}$ for all $n \in \mathbb{N}$. 
	Let $\sigma:\mathbb{N} \to \mathbb{N}$ be an infinite-periodic function with $\Gamma_\infty = \{1,\ldots,r\}$. 
	Let $\{T_n\}$ be defined by \eqref{eq define Tn}, and let $x \in \mathscr{H}$ be arbitrary. 
	Then the weak limit of any weakly convergent subsequence of $\{T_n x\}$ belongs to $\bigcap_{i=r+1}^\infty \mathscr{M}_i$.
\end{lemma}

\begin{proof}
	Let $I(\sigma,i) = m_i$ for all $i \geq r+1$. In view of  Lemma \ref{sakai}, $\lim_{n\to\infty}\mathrm{dist}(T_nx,\mathscr{M}_i)=0$ for each $i\geq r+1$. Indeed, let $i\geq r+1$ be fixed. We have
	\begin{align*}
		\mathrm{dist}(T_nx,\mathscr{M}_i)&=\|T_nx-P_iT_nx\|\\
		&\leq \|T_nx-T_{n+s_i}x\|+\|T_{n+s_i}x-P_iT_nx\|\\
		&\qquad\qquad\qquad\qquad(\mbox{where}~0\leq s_i\leq m_i\mbox{~are such that~} P_{\sigma(n+s_i)}=P_i)\\
		&= \|T_nx-T_{n+s_i}x\|+\|P_iT_{n+s_i-1}x-P_iT_nx\|\\
		&\leq \|T_nx-T_{n+s_i}x\|+\|T_{n+s_i-1}x-T_nx\|.
	\end{align*} 
	This, together with Lemma \ref{sakai}, proves the claim. Hence, there are $y_n\in\mathscr{M}_i$ such that 
	\begin{equation}\label{lemma distance}
		\lim_{n \to \infty}\|T_nx-y_n\|=0
	\end{equation}
	To simplify notation, we denote the weakly convergent subsequence by $\{T_n x\}$ itself, with weak limit $x'$. Employing \eqref{lemma distance} gives that $x'$ is the weak limit of the sequence $\{y_n\}$. Therefore, $x'\in \mathscr M_i$.	
	Since $i \geq r+1$ was arbitrary, we conclude that $x' \in \bigcap_{i=r+1}^\infty \mathcal{M}_i$.
\end{proof}
\begin{theorem}\label{thempty}
	Let $P_n \in \mathbb{B}(\mathscr{H})$ be projections onto closed subspaces $\mathscr{M}_n\subseteq\mathscr{H}$ for all $n \in \mathbb{N}$. 
	Let $\sigma:\mathbb{N} \to \mathbb{N}$ be an infinite-periodic function such that $\Gamma_\infty = \emptyset$. 
	Set $T_1 := P_{\sigma(1)}$, $T_n := P_{\sigma(n)} T_{n-1}$ for all $n \ge 2$. For an arbitrary $x \in \mathscr{H}$, the sequence $\{T_nx\}$ 
	converges weakly to $Px$, where $P$ is the projection onto $\bigcap_{i \in J} \mathcal{R}(P_i)$.
\end{theorem}
\begin{proof}
	According to Lemma \ref{lemma weak}, it is enough to show that each weakly convergent subsequence of $\{T_nx\}$ has the weak limit $Px$. Suppose that $\{T_{n_i}x\}$ is a subsequence of $\{T_nx\}$ such that $w\!-\!\lim_{i\to\infty} T_{n_i} x=x_0$. Employing Lemma \ref{lemma p} gives us $x_0\in \bigcap_{i=1}^\infty \mathcal{R}(P_i)$, so $Px_0=x_0$. Now, we have 
	\[
	Px=w\!-\!\lim_{i\to\infty} PT_{n_i} x=Pw\!-\!\lim_{i\to\infty} T_{n_i} x=Px_0=x_0.
	\]
	Hence, $w\!-\!\lim_{i\to\infty} T_{n_i} x=Px$.
\end{proof}

\begin{lemma}\label{lemma k_n}
	Let $P_j \in \mathbb{B}(\mathscr{H})$ be projections onto closed subspaces $\mathscr{M}_j$ of $\mathscr{H}$ for all $j \in \mathbb{N}$. 
	Let $\sigma : \mathbb{N} \to \mathbb{N}$ be an infinite-periodic function with $\Gamma_{\infty} = \{1, \ldots, r\}$. 
	Let $\{T_n\}_{n=1}^\infty$ and $\{k_n\}_{n=1}^\infty$ be defined by \eqref{eq define Tn}, and let $x \in \mathscr{H}$ be arbitrary. Let $\{{r_n}\}\subseteq\mathbb{N}$ be a sequence.
	If $\{T_{k_{r_n}} x\}$ converges weakly to $x_0$, then for every $\ell \ge 1$, the sequence $\{T_{k_{r_{n}+\ell}} x\}_{n=1}^\infty$ also converges weakly to $x_0$.
\end{lemma}

\begin{proof}
	We first prove the assertion for $\ell=1$. Suppose that $\displaystyle w\!-\!\lim_{n \to \infty} T_{k_{r_n}} x = x_0$. 
	From Lemma \ref{lemma p}, we infer that $x_0 \in \bigcap_{j \in \Gamma_F} \mathcal{R}(P_j)$. 
	Let $Q$ be the projection onto $\bigcap_{j \in \Gamma_F} \mathcal{R}(P_j)$. Then $Q x_0 = x_0$, and consequently
	\begin{align}\label{eqr_n+1}
		w\!-\!\lim_{n \to \infty} Q T_{k_{r_n+1}-1} x&= w\!-\!\lim_{n \to \infty} Q P_{\sigma(k_{r_n+1}-1)}P_{\sigma(k_{r_n+1}-2)}\cdots P_{\sigma(k_{r_n}+1)}T_{k_{r_n}} x\nonumber\\
		&= w\!-\!\lim_{n \to \infty} Q T_{k_{r_n}} x\quad\quad(\mbox{since~}QP_j=Q \mbox{~for~} j\in \Gamma_F)\nonumber\\
		&= Q w\!-\!\lim_{n \to \infty} T_{k_{r_n}} x \nonumber\\
		& = x_0.
	\end{align}
	
	Now let $\{T_{k_{s_n+1}} x\}$ be an arbitrary weakly convergent subsequence of $\{T_{k_{r_n+1}} x\}$, with weak limit $x'$. 
	Lemma \ref{lemma p} ensures that $x' \in \bigcap_{j \in \Gamma_F} \mathcal{R}(P_j)$. On the other hand, by Lemma \ref{sakai}, we have
	\[
	x'= w\!-\!\lim_{n \to \infty} T_{k_{s_n+1}-1} x.
	\] 
	Employing \eqref{eqr_n+1} gives us 
	\begin{align*}
		x' = Q x' &= w\!-\!\lim_{n \to \infty} Q T_{k_{s_n+1}-1} x\\
		&=w\!-\!\lim_{n \to \infty} Q T_{k_{r_n+1}-1} x\quad(\mbox{since~}\{s_{n}\}~\mbox{is~a~subsequence~of ~}\{r_{n}\})\\
		&= x_0.
	\end{align*}
	
	Lemma \ref{lemma weak} yields that $\{T_{k_{r_n+1}} x\}$ converges weakly to $x_0$. 
	Now let $w\!-\!\lim_{n \to \infty} T_{k_{r_n+\ell}} x = x_0$. Applying the same technique as above gives $w\!-\!\lim_{n \to \infty} T_{k_{r_n+\ell+1}} x = x_0$ for every $\ell \ge 1$. Now, by induction on $\ell$, we obtain the result.
\end{proof}

Let $\sigma:\mathbb{N}\to \mathbb{N}$ be an infinite-periodic function such that $\Gamma_{\infty}=\{1,\ldots,r\}$ for some $r\geq 2$, and let $\{k_n\}$ be as in Definition \ref{pse}. Since $\sigma^{-1}(j)$ is an infinite set for all $j\in \mathbb{N}$ and $P_{\sigma(k_n)}\in \{P_1,\ldots ,P_r\}$, so for each $k_n$ there exist $s,t\geq 0$ (depending on $n$) such that 
\begin{align}\label{eqleft}
	P_{\sigma(k_{n-s})}=P_{\sigma(k_{n-s+1})}=\cdots&=P_{\sigma(k_n)}=\cdots =P_{\sigma(k_{n+t-1})}=P_{\sigma(k_{n+t})},\nonumber\\ P_{\sigma(k_{n-s-1})}&\neq P_{\sigma(k_n)}\neq P_{\sigma(k_{n+t+1})}. 
\end{align}
In fact, 
\[
s=\min\{p: \mbox{for~all~}m\in [p,n],P_{\sigma(k_m)}=P_{\sigma(k_n)}\},\] and
\[
t=\max\{q: \mbox{for~all~}m\in [n,q],P_{\sigma(k_m)}=P_{\sigma(k_n)}\}
\]
Suppose that $\{T_{k_{r_n}} x\}$ is a subsequence of $\{T_{k_n}x\}$. We denote the {\emph{ left (right)-shift sequence} of $\{k_{r_n}\}$ by $\{k_{s_n}\} (\{k_{t_n}\})$, where $s_n$ (and $t_n$) satisfy \eqref{eqleft} by replacing $n$ with $r_n$. Hence, for each $n\in \mathbb{N}$ and each $s_n\leq m\leq t_n$, we have 
	\begin{equation}\label{eqleftshift}
		P_{\sigma(k_m)}=P_{\sigma(k_{r_n})}\quad \mbox{and} \quad P_{\sigma(k_{s_n-1})}\neq P_{\sigma(k_{r_n})}\neq P_{\sigma(k_{t_n+1})}.
	\end{equation} 
	\begin{example}
		Let $\{P_i : i \geq 1\}$ be a sequence of projections on a Hilbert space $\mathscr{H}$. Let $S=\{2^j+1; j\geq 0\}$. We define $\sigma : \mathbb{N} \to \mathbb{N}$ by:
		
		\begin{align*}
			\sigma(n):= \begin{cases}
				i+2& n\not\in S,\; n= 2^{i-1}(2k-1), \; i,k \in \mathbb{N}\\
				1& n=2^j+1, \; j\in [m(m+1),m(m+2)], \; m\geq 0\\
				2& \text{otherwise}
			\end{cases}.
		\end{align*}
		
		By the same method as used in Example \ref{exam}, we see that $\sigma$ is an infinite-periodic function with $\Gamma_{\infty}=\{1,2\}$ and $k_n=2^n+1$, $n\geq 0$. In this case we have
		\[
		\begin{matrix}
			\cdots & P_{\sigma(2^8+1)} & P_{\sigma(2^7+1)} & P_{\sigma(2^6+1)} & P_{\sigma(2^5+1)} & P_{\sigma(2^4+1)} & P_{\sigma(9)} & P_{\sigma(5)} & P_{\sigma(3)} & P_{\sigma(1)} \\
			~&\shortparallel&\shortparallel&\shortparallel&\shortparallel&\shortparallel&\shortparallel&\shortparallel&\shortparallel&\shortparallel\\
			\cdots & P_1 & P_1 & P_1 & P_2 & P_2 & P_1 &P_1 & P_2 & P_1 
		\end{matrix}
		\]
		Let $r_n$ be the last place that $\sigma(k_{r_n})=1$ in each period $[m(m+1),m(m+2)]$, that is, $1,2^3+1, 2^{8}+1,\cdots, 2^{n(n+2)},\cdots$. Then the left-shift sequence $\{k_{s_n}\}$ is $1,2^2+1,2^{6}+1,2^{12}+1,2^{20}+1 $, that is, $\{k_{s_n}\}=\{2^{n(n+1)}+1\}_{n=0}^\infty$. Note that in this case $\{k_{t_n}\}$, the right-shift sequence of $\{k_{r_n}\}$, is $\{k_{r_n}\}$. 	Let $\{T_n\}$ be defined by \eqref{eq define Tn}, and let $x \in \mathscr{H}$ be arbitrary. In this case $\Gamma_\infty=\{1,2\}$ , and $\Gamma_F=\{P_j; j\geq 3\}$. Suppose that $Q$ is the projection onto $\bigcap_{j \in \Gamma_F} \mathcal{R}(P_j) \cap \mathcal{R}(P_1)$. Since, $QP_j=Q$ for all $j\in \Gamma_F\cup\{1\}$, we have 
		\[
		QT_{k_{r_n}}=QT_{k_{s_n}}.
		\]
	\end{example}
	We need the following lemma to achieve the main result. 
	\begin{lemma}\label{lemma shift}
		Let $P_n \in \mathbb{B}(\mathscr{H})$ be projections onto closed subspaces $\mathscr{M}_n$ of $\mathscr{H}$ for all $n \in \mathbb{N}$. 
		Let $\sigma : \mathbb{N} \to \mathbb{N}$ be an infinite-periodic function with $\Gamma_{\infty} = \{1, \ldots, r\}$, where $r\geq 2$. 
		Let $\{T_n\}$ be defined by \eqref{eq define Tn}. Let $ \{k_n\} $ be the sequence from Definition \ref{pse}.	Suppose that $\{T_{k_{r_n}} x\}$ converges weakly to $x_0$ for some subsequence $\{k_{r_n}\}$ of $\{k_n\}$. 
		Then  $\{T_{k_{s_n}}x\}$ ($\{T_{k_{t_n}}x\}$) also converges weakly to $x_0$, where $\{k_{s_n}\}$ ($\{k_{t_n}\}$) is the left (right) shift sequence of $\{k_{r_n}\}$. 
		
	\end{lemma}
	
	\begin{proof}	
		We only prove the case of the left-shift sequence, because ${k_{s_n}}$ is the left-shift sequence of ${k_{t_n}}$. From the proof of the left-shift case, the weak limit of any weakly convergent subsequence of $T_{k_{t_n}}x$ is $x_0$. It follows from Lemma \ref{lemma weak} that $w\!-\lim_{n \to \infty} T_{k_{t_n}} x = x_0$.
		
		Since $\{P_{\sigma(k_{r_n})} : n \ge 1\}$ is contained in $\{P_1, \ldots, P_r\}$, we can partition $\{T_{k_{r_n}}\}_{n=1}^{\infty}$ into finitely many subsequences $\{T_{k_{r'_n}}\}_{n=1}^{\infty}$ such that $P_{\sigma(k_{r'_n})} = P_{\sigma(k_{r'_m})}$ for all $n, m \ge 1$. 
		Hence, without loss of generality, we may assume $P_{\sigma(k_{r_n})} = P_{r_0}$ for all $n \ge 1$, where $1 \le r_0 \le r$.
		
		Let $w\!-\!\lim_{n \to \infty} T_{k_{r_n}} x = x_0$. 
		It follows from Lemma \ref{lemma p} that $x_0 \in \bigcap_{j \in \Gamma_F} \mathcal{R}(P_j)$. By the assumption that $P_{\sigma(k_{r_n})}=P_{r_0}$, we also have $x_0 \in \mathcal{R}(P_{r_0})$, so $x_0\in \bigcap_{j \in \Gamma_F} \mathcal{R}(P_j) \cap \mathcal{R}(P_{r_0})$. 
		Let $Q$ be the projection onto $\bigcap_{j \in \Gamma_F} \mathcal{R}(P_j) \cap \mathcal{R}(P_{r_0})$. Then
		\begin{align}\label{eq x_0}
			x_0 = Q x_0 
			&= Q \bigl( w\!-\!\lim_{n \to \infty} T_{k_{r_n}} x \bigr) \nonumber\\
			&= w\!-\!\lim_{n \to \infty} Q T_{k_{r_n}} x \nonumber\\
			&= w\!-\!\lim_{n \to \infty} Q P_{\sigma(k_{r_n})}P_{\sigma(k_{r_n}-1)}\cdots P_{\sigma(k_{s_n}+1)} T_{k_{s_n}} x \nonumber\\
			&= w\!-\!\lim_{n \to \infty} Q T_{k_{s_n}}x \qquad\quad(\mbox{since~}QP_j=Q,\mbox{~for~}j\in\Gamma_F\cup\{r_0\}).
		\end{align}
		
		If we show that the weak limit of all weakly convergent subsequences of $\{T_{k_{s_{n}}}x\}$ equals $x_0$, then the result follows from Lemma \ref{lemma weak}. Let $x'$ be a weak limit of a subsequence of $\{T_{k_{s_{n}}}x\}$ which we again denote by $\{T_{k_{s_{n}}}x\}$, that is, $x' = w\!-\!\lim_{n \to \infty} T_{k_{s_{n}}} x$. 
		From the hypothesis on $\{T_{k_{r_n}}\}$, we have 
		\begin{align*}
			T_{k_{s_n}}&=P_{\sigma(k_{s_n})}T_{k_{s_n}}\\
			&=P_{\sigma(k_{r_n})}T_{k_{s_n}}\qquad\qquad(\mbox{by~}\eqref{eqleftshift})\\
			&=P_{r_0}T_{k_{s_n}},
		\end{align*}
		whence we obtain $x' \in \mathcal{R}(P_{r_0})$. 
		Lemma \ref{lemma p} also gives $x' \in \bigcap_{j \in \Gamma_F} \mathcal{R}(P_j)$. 
		Hence $x' \in \bigcap_{j \in \Gamma_F} \mathcal{R}(P_j) \cap \mathcal{R}(P_{r_0})=\mathcal{R}(Q)$, and therefore, by \eqref{eq x_0} we have
		\[
		x' = Q x' = w\!-\!\lim_{n \to \infty} Q T_{k_{s_n}} x = x_0.
		\]
		Thus $\displaystyle w\!-\!\lim_{n \to \infty} T_{k_{s_n}} x = x_0$.

	\end{proof}
	We can generalize the concept of the left-shift subsequence to two-fold and 
	$\ell$-fold projections, for example in the case $\ell=2$, let $r\geq 3$ and let $P_s,P_t\in \{P_1,\ldots,P_r\}$. Let us define $s_n$ and $t_n$ with $s_n\leq r_n\leq t_n$ as follows: 
	\begin{enumerate}
		\item if $P_{\sigma(k_{r_n})}\not\in \{P_s,P_t\}$, then
		\begin{align*}
			s_n:=t_n:=r_n;
		\end{align*}
		\item if $P_{\sigma(k_{r_n})}\in \{P_s,P_t\}$, then
		\begin{align*}
			s_n:=\min\left\{p:\mbox{\text{for~all~}} m\in [p,r_n], P_{\sigma(k_m)}\in \{P_s,P_t\}\right\}
		\end{align*}
		and
		\begin{align*}
			t_n:=\max\left\{q:\mbox{\text{for~all~}} m\in [r_n,q], P_{\sigma(k_m)}\in \{P_s,P_t\}\right\}.
		\end{align*}
	\end{enumerate}
	We now show that Lemma \ref{lemma shift} holds for any $\ell$ projections with $\ell < r$.
	\begin{theorem}\label{remark1}
		
		Let $P_j \in \mathbb{B}(\mathscr{H})$ be projections onto closed subspaces $\mathscr{M}_j$ of $\mathscr{H}$ for all $j \in \mathbb{N}$. 
		Let $\sigma : \mathbb{N} \to \mathbb{N}$ be an infinite-periodic function with $\Gamma_{\infty} = \{1, \ldots, r\}$. 
		Let $\{T_n\}$ and $\{k_n\}$ be defined by \eqref{eq define Tn}. Let $ \{k_n\} $ be the sequence as in Definition \ref{pse}.	Suppose that $\{T_{k_{r_n}} x\}$ converges weakly to $x_0$ for some subsequence $\{{r_n}\}\subseteq\mathbb{N}$. Let $\ell<r$ be arbitrary. Let us define $s_n$ and $t_n$ with $s_n\leq r_n\leq t_n$ as follows: 
		\begin{enumerate}
			\item if $P_{\sigma(k_{r_n})}\not\in \{P_1,\ldots,P_\ell\}$, then
			\begin{equation}\label{eq i}
				s_n:=t_n:=r_n;
			\end{equation}
			\item if $P_{\sigma(k_{r_n})}\in \{P_1,\ldots,P_\ell\}$, then
			\begin{align}\label{eq two}
				s_n:=\min\left\{p:\mbox{\text{for~all~}} m\in [p,r_n], P_{\sigma(k_m)}\in \{P_1,\ldots,P_\ell\}\right\}
			\end{align}
			and
			\begin{align}\label{eq two1}
				t_n:=\max\left\{q:\mbox{\text{for~all~}} m\in [r_n,q], P_{\sigma(k_m)}\in \{P_1,\ldots,P_\ell\}\right\}.
			\end{align}
		\end{enumerate}
		Then $\{T_{k_{s_n}} x\}$ and $\{T_{k_{t_n}}x\}$ weakly converge to $x_0$. 
	\end{theorem}	 
	\begin{proof}
		For $\ell = 1$, $\{T_{k_{s_n}}x\}$ and $\{T_{k_{r_n}}x\}$ are left-shift and right-shift sequences of $\{T_{k_{r_n}}x\}$ respectively. Hence,  Lemma \ref{lemma shift} gives that $w\!-\!\lim_{n \to \infty} T_{k_{s_n}} x = x_0=w\!-\!\lim_{n \to \infty} T_{k_{t_n}} x$.
		
		Assume now the statement holds for collection of at most $\ell-1$ projections. For each $\ell_0\leq \ell-1$, let us define indices $s'_n$ and $t'_n$ with $s'_n\leq r_n\leq t'_n$ as follows: 
		\begin{enumerate}
			\item if $P_{\sigma(k_{r_n})}\not\in \{P_1,\ldots,P_{\ell_0}\}$, then
			\begin{align*}
				s'_n:=t'_n:=r_n;
			\end{align*}
			\item if $P_{\sigma(k_{r_n})}\in \{P_1,\ldots,P_{\ell_0}\}$, then
			\begin{align*}
				s'_n:=\min\left\{p:\mbox{\text{for~all~}} m\in [p,r_n], P_{\sigma(k_m)}\in \{P_1,\ldots,P_{\ell_0}\}\right\}
			\end{align*}
			and
			\begin{align*}
				t'_n:=\max\left\{q:\mbox{\text{for~all~}} m\in [r_n,q], P_{\sigma(k_m)}\in \{P_1,\ldots,P_{\ell_0}\}\right\}.
			\end{align*}
		\end{enumerate}
		Then 
		\begin{equation}\label{eqinduction}
			w\!-\!\lim_{n \to \infty} T_{k_{s'_n}} x=w\!-\!\lim_{n \to \infty} T_{k_{t'_n}} x=x_0.
		\end{equation}
		Since $\sigma^{-1}(j)$ is an infinite set for each $j$, and since $\ell < r$, the sequence $\{s_n\}$ is increasing. 
		For each $n\geq 1$, set
		\[
		S_n=\{P_{\sigma(k_m)}:s_n\leq m\leq r_n\}.
		\]
		If there exists $m\geq 1$ such that $S_n\subsetneq\{P_1,\ldots,P_\ell\}$ for each $n\geq m$, then, the sequence 
		$\{T_{k_{s_n}}x\}$
		coincides with the sequence 
		$\{T_{k_{s'_n}}x\}$
		that appears in the inductive hypothesis for some $\ell_0<\ell$. Therefore, by \eqref{eqinduction} we obtain $w\!-\!\lim_{n \to \infty} T_{k_{s_n}} x=x_0$. Thus, without loss of generality, we may assume that $S_n=\{P_1,\ldots,P_\ell\}$ for all $n\geq 1$. Let $Q$ be the projection onto 
		\[
		\bigcap_{j \in \Gamma_F} \mathcal{R}(P_j)\cap \mathcal{R}(P_1)\cap\cdots \cap \mathcal{R}(P_\ell).
		\] We first show that $x_0\in\mathcal{R}(Q)$. Indeed, if for each $1\leq \ell_0\leq \ell$, the set $\{n:P_{\ell_0}=P_{\sigma(k_{r_n})}\}$ is infinite, then $x_0\in \mathcal{R}(Q)$. Otherwise, after re-indexing, let $\ell'$ ($\ell'<\ell$)
		be the largest index such that the sets $\{n:P_{\ell_0}=P_{\sigma(k_{r_n})}\}$ are infinite for all $1\leq \ell_0\leq \ell'$. Then $x_0\in \bigcap_{j \in \Gamma_F} \mathcal{R}(P_j)\cap \mathcal{R}(P_1)\cap\cdots \cap \mathcal{R}(P_{\ell'})$.
		Choose $\{T_{k_{s'_n}}x\}$ such that $s'_n$ satisfies condition (2) of the induction hypothesis. Then, $w\!-\!\lim_{n \to \infty} T_{k_{s'_n}} x=x_0$. Since, $S_n=\{P_1,\ldots,P_\ell\}$ for all $n\geq 1$, we have $P_{\sigma(k_{s'_n-1})}\not\in \{P_1,\ldots,P_{\ell'}\}$ for infinitely many $n\geq 1$. Hence there exists $\ell'<\ell''\leq \ell$ such that $P_{\sigma(k_{s'_n-1})=}P_{\ell''}$ for infinitely many $n$. Without loss of generality, take $\ell''=\ell'+1$. Moreover, in light of Lemma \ref{lemma k_n}, we have $w\!-\!\lim_{n \to \infty} T_{k_{s'_n-1}} x=x_0$. Therefore, $x_0\in \mathcal{R}(P_{\ell'+1})$. 
		
		Replacing $\{P_1,\ldots,P_{\ell'}\}$ by $\{P_1,\ldots,P_{\ell'+1}\}$ and applying the induction hypothesis again yields $x_0\in \mathcal{R}(P_{\ell'+2})$. Repeating this argument gives $x_0\in \bigcap_{j \in \Gamma_F} \mathcal{R}(P_j)\cap \mathcal{R}(P_1)\cap\cdots \cap \mathcal{R}(P_r)=\mathcal{R}(Q)$

		Now, let $x'$ be a weak limit of any subsequence of $\{T_{k_{s_n}}x\}$. Without loss of generality, assume $x' = w\!-\!\lim_{n \to \infty} T_{k_{s_n}} x$.
		
		We claim that $x'\in \mathcal{R}(Q)$. According to Lemma \ref{lemma p}, we have $x' \in \bigcap_{j \in \Gamma_F} \mathcal{R}(P_j)$.
		Let $\ell_0\leq \ell$ be arbitrary. Let $\{T_{k_{t'_n}}x\}$ be the sequence from the induction hypothesis corresponding to the projections $\{P_1,\ldots,P_{\ell}\}\setminus\{\ell_0\}$. After a suitable relabeling, we may assume $\ell_0=\ell$, so that the sequence is associated with $\{P_1,\ldots,P_{\ell-1}\}$. By the induction hypothesis, $ w\!-\!\lim_{n \to \infty} T_{k_{t'_n}} x=x'$. Since $S_n=\{P_1,\ldots,P_\ell\}$ for all $n$, we have $P_{\sigma(k_{{t'_n}+1})}=P_{\ell}$ for infinitely many $n$. Lemma \ref{lemma k_n} then gives that $ w\!-\!\lim_{n \to \infty} T_{k_{t'_n+1}} x=x'$, so $x'\in \mathcal{R}(P_\ell)=\mathcal{R}(P_{\ell_0})$. As $\ell_0\leq \ell$ was arbitrary, it follows that $x'\in \mathcal{R}(Q)$.
		From \eqref{eq two}, for each $j\in [s_n,r_n]$ we have $QP_{\sigma(j)}=Q$. Therefore, 
		\[
		QT_{k_{r_n}}=QP_{\sigma(k_{r_n})}P_{\sigma(k_{r_n}-1)}\cdots P_{\sigma(k_{s_n}+1)}T_{k_{s_n}}=QT_{k_{s_n}}.
		\] Consequently,
		\begin{align*}\label{eq x_0 1}
			x_0 = Q x_0 
			= Q \bigl( w\!-\!\lim_{n \to \infty} T_{k_{r_n}} x \bigr)
			= w\!-\!\lim_{n \to \infty} Q T_{k_{r_n}} x 
			=w\!-\!\lim_{n \to \infty} Q T_{k_{s_n}} x
			= Qx' 
			= x'.
		\end{align*}
	\end{proof}

	Now we are ready to present the main result.
	
	\begin{theorem}\label{main th}
		Let $P_n \in \mathbb{B}(\mathscr{H})$ be projections onto closed subspaces $\mathscr{M}_n$ of $\mathscr{H}$ for all $n \in \mathbb{N}$. 
		Let $\sigma : \mathbb{N} \to \mathbb{N}$ be an infinite-periodic function  such that $\Gamma_\infty=\{1,\cdots, r\}$  for some $r\in \mathbb{N}$.
		Set $T_1 := P_{\sigma(1)}$, $T_n := P_{\sigma(n)} T_{n-1}$ for all $n \ge 2$. 
		Then, $\{T_n x\}$ converges weakly to $Px$, where $P$ is the projection onto $\bigcap_{i \in \mathbb{N}} \mathcal{R}(P_i)$.
	\end{theorem}
	
	\begin{proof} 
		By the definition of an infinite-periodic function let $I(\sigma,j) = m_j < \infty$ for all $j \ge r+1$.

		Let $\{T_{n_i} x\}$ be an arbitrary weakly convergent subsequence of $\{T_n x\}$ with weak limit $x_0$. 
		It follows from Lemma \ref{lemma p} that $x_0 \in \bigcap_{j \in \Gamma_F} \mathcal{R}(P_j)$. 
		Let $Q$ be the projection onto $\bigcap_{j \in \Gamma_F} \mathcal{R}(P_j)$. Then $Q x_0 = x_0$ and $\{Q T_{n_i} x\}$ converges weakly to $x_0$. 
		
		Note that
		\[
		Q T_{n_i} x = Q T_{k_{r_i}} x,
		\]
		where $k_{r_i} \le n_i < k_{r_i+1}$ for each $i\geq 1$. We represent sequence $\{r_i\}$ with $\{r_n\}$.
		Since $\{T_{k_{r_n}} x\}$ is bounded, it has a weakly convergent subsequence; without loss of generality we may assume $\{T_{k_{r_n}} x\}$ itself converges weakly to some $x_1$. 
		Then $Q x_1 = x_0$. Again Lemma \ref{sakai} yields that $x_1 \in \bigcap_{j \in \Gamma_F} \mathcal{R}(P_j)$, so $x_1 = Q x_1 = x_0$.
		
		Let
		\[
		R = \{ P_s : P_s = P_{\sigma(k_{r_n})} \text{ for infinitely many } n \} \subseteq \{P_1, \ldots, P_r\}.
		\]
		We consider two cases:

		\textbf{Case A.} If $R=\{P_1,P_2,\ldots,P_r\}$, then $x_0\in \bigl( \bigcap_{j \in \Gamma_F} \mathcal{R}(P_j) \bigr) \cap \bigl( \bigcap_{j=1}^r \mathcal{R}(P_j) \bigr)$. So, $Px_0=x_0$ and hence, 
		\begin{align}\label{eq r1}
			w\!-\!\lim_{i \to \infty} T_{{n_i}} x &=w\!-\!\lim_{n \to \infty} T_{k_{r_n}} x =x_0=Px_0\nonumber\\
			&=Pw\!-\!\lim_{i \to \infty} T_{n_i} x =w\!-\!\lim_{i \to \infty} PT_{n_i} x\nonumber\\
			&=Px.	
		\end{align}
		
		\textbf{Case B.} If $R=\{P_1,P_2,\ldots, P_{r_0}\}$, where $r_0<r$, then $x_0 \in \bigl( \bigcap_{j \in \Gamma_F} \mathcal{R}(P_j) \bigr) \cap \bigl( \bigcap_{j \in R} \mathcal{R}(P_j) \bigr)$. 
		Let $Q$ be the projection onto $\bigl( \bigcap_{j \in \Gamma_F} \mathcal{R}(P_j) \bigr) \cap \bigl( \bigcap_{j \in R} \mathcal{R}(P_j) \bigr)$. Suppose that $T_{k_{s_n}}$ is as in Theorem \ref{remark1} with respect to $\ell=r_0$, that is, $QP_{\sigma(j)}=Q$ for all $s_n\leq j\leq r_n$, and we have $P_{\sigma(k_{s_n-1})}\not\in \{P_1,P_2,\cdots, P_{r_0}\}$ for infinitely many $n$. We therefore have 
		\[
		Q T_{k_{r_n}} = Q T_{k_{s_n}}.
		\]

		By Theorem \ref{remark1}, we have that $w\!-\!\lim_{n \to \infty} T_{k_{s_n}} x=x_0$. In addition, by Lemma \ref{lemma k_n}, we have $w\!-\!\lim_{n \to \infty} T_{k_{s_n-1}} x=x_0$. Since $P_{\sigma(k_{s_n-1})}\not\in \{P_1,\cdots, P_{r_0}\}$, thus there exists $r_0<r'\leq r$ such that $P_{\sigma(k_{s_n-1})}=P_{r'}$ for infinitely many $n$, and hence, $x_0\in \mathcal{R}(P_{r'})$. By repeating this process and  replacing $\{T_{k_{r_n}}x\}$ with $\{T_{k_{s_n-1}}x\}$, we obtain $x_0\in \bigcap_{j\in \Gamma_{\infty}}\mathcal{R}(P_j)$. This ensures, in the same way as in \eqref{eq r1}, that $w\!-\!\lim_{n \to \infty} T_{k_{r_n}} x=Px$. 
		Therefore the weak limit of any weakly convergent subsequence of $\{T_n x\}$ lies in $\bigcap_j \mathcal{R}(P_j)$, and in view of Lemma \ref{lemma weak}, the whole sequence converges weakly to $Px$.
	\end{proof}
	
	\section{Concluding remarks}
	
	We start this section with an example of an infinite-periodic function.
	
	\begin{example}\label{exam1}
		Let $\{P_i : i \geq 1\}$ be a sequence of projections on a Hilbert space $\mathscr{H}$. Define $\sigma : \mathbb{N} \to \mathbb{N}$ by
		\begin{equation}\label{eqexampleg11}
			\sigma(n):= i, \quad \bigl( n = 2^{i-1}(2k-1), \mbox{where~} i,k \in \mathbb{N} \bigr).
		\end{equation}
		For instance, for $i = 1$, using \eqref{eqexampleg11}, the set of all $l_n$ satisfying $\sigma(l_n) = 1$ is
		$\{1, 3, 5, 7, \ldots\}$.
		For $i = 2$, the set of all $l_n$ satisfying $\sigma(l_n) = 2$ is $\{2, 6, 10, 14, \ldots\}$.
		From \eqref{eqexampleg11} we obtain, for each $i \geq 1$,
		\[
		I(\sigma,1) = 2,\; I(\sigma,2) = 4,\; I(\sigma,3) = 8,\; \ldots,\; I(\sigma,i) = 2^i .
		\]
		Hence $\sigma$ is an infinite-periodic function. In this case,
		\[
		\Gamma_F = \{ i : i \geq 1 \}\quad \mbox{and} \quad \Gamma_\infty = \emptyset .
		\]

	\end{example}
	\begin{remark}\label{lemma strong}
		Now let $\sigma$ be an infinite-periodic function defined in Example \ref{exam1}. Fix $x \in \mathscr H$ and define
		\[
		T_1 x = P_{\sigma(1)} x\quad \mbox{and} \quad T_n x = P_{\sigma(n)} T_{n-1} x \quad (n \geq 2).
		\]
		
		Employing Theorem \ref{main th} gives $w\!-\!\lim_{n\to \infty} T_n x = Px$. 
		
		Let $n \ge 1$, $1 \le k < 2^n$, and write $k = 2^{i-1}(2r-1)$. Then
		\[
		2^n - k = 2^n - 2^{i-1}(2r-1) = 2^{i-1}(2^{n-i+1} - 2r + 1) = 2^{i-1}(2s-1),
		\]
		where $s = 2^{n-i} - r + 1$. Hence $\sigma(2^n - k) = \sigma(k)$ for all $n \ge 1$ and for all $1 \le k < 2^n$. 
		Similarly $\sigma(2^n + k) = \sigma(k)$. Consequently for all $n\geq 1$
		\begin{align}\label{eq T_n}
			T_{2^{n+1}-1} &=P_{\sigma(2^{n}+2^n-1)}\ldots P_{\sigma(2^{n}+2)} P_{\sigma(2^{n}+1)} T_{2^n}\nonumber\\
			&=P_{\sigma(2^n-1)}\cdots P_{\sigma(2)}P_{\sigma(1)}T_{2^n}\nonumber\\
			& =T_{2^{n}-1}T_{2^n}, 
		\end{align}
		and
		\begin{align}\label{eq T^*_n}
			T_{2^n-1}^* &= \left(P_{\sigma(2^n-1)}P_{\sigma(2^n-2)}\cdots P_{\sigma(2)}P_{\sigma(1)}\right)^*\nonumber\\
			&=P_{\sigma(2^n-(2^n-1))}P_{\sigma(2^n-(2^n-2))}\cdots P_{\sigma(2^n-2)}P_{\sigma(2^n-1)}\nonumber\\
			&=P_{\sigma(2^n-1)}P_{\sigma(2^n-2)}\cdots P_{\sigma(2)} P_{\sigma(1)}= T_{2^n-1},
		\end{align}
		where $^*$ denotes that the adjoint operation.	Therefore
		\begin{align*}
			\|Px\|^2 &= \lim_{n\to\infty} \langle T_{2^{n+1}-1} x, x \rangle \\
			&= \lim_{n\to\infty} \langle T_{2^n-1} T_{2^n} x, x \rangle\qquad\qquad(\text{by~}\eqref{eq T_n}) \\
			&= \lim_{n\to\infty} \langle T_{2^n} x, T^*_{2^n-1} x \rangle \\
			&= \lim_{n\to\infty} \langle T_{2^n} x, T_{2^n-1} x \rangle\qquad\qquad(\text{by~}\eqref{eq T^*_n}) \\
			&= \lim_{n\to\infty} \langle P_{\sigma(2^n)}T_{2^n} x, T_{2^n-1} x \rangle \qquad(\text{since }P_{\sigma(2^n)}\text{ is a projection})\\
			&= \lim_{n\to\infty} \langle T_{2^n} x, P_{\sigma(2^n)}T_{2^n-1} x \rangle\\
			&= \lim_{n\to\infty} \langle T_{2^n} x, T_{2^n} x \rangle \\
			&= \lim_{n\to\infty} \|T_{2^n} x\|^2.
		\end{align*}
		This shows that $\{T_n x\}$ converges strongly to $Px$.
	\end{remark}
	
	\begin{corollary}
		Let $\sigma:\mathbb{N}\to \mathbb{N}$ be an infinite-periodic function.
		Suppose there exists a subsequence $\{r_n\}$ of $\mathbb{N}$ such that each $T_{r_n}$ is positive, where $\{T_n\}$ is defined by \eqref{eq define Tn}. 
		Then $\{T_n x\}$ converges strongly to $Px$, where $P$ is the projection onto $\bigcap_{i \in J} \mathcal{R}(P_i)$.
	\end{corollary}
	
	\begin{proof}
		Since $\|T_{r_n+1} x\| \le \|T_{r_n} x\|$, we have $\|T_{r_n+1} x\|^2 \le \|T_{r_n} x\|^2$, hence $T_{r_n+1}^2 \le T_{r_n}^2$. 
		This implies $T_{r_n+1} \le T_{r_n}$. By the Vigier theorem \cite[Theorem 4.1.1]{Mur}, $\{T_{r_n}\}$ converges strongly to an operator $Q$. 
		Theorem \ref{main th} yields $Q = P$.
	\end{proof}
	
	
	\medskip
	\noindent \textit{Author Contributions Statement.} All authors wrote the main manuscript text and they edited and reviewed the manuscript.
	
	\medskip
	\noindent \textit{Conflict of Interest Statement.} On behalf of the authors, the corresponding author states that there is no conflict of interest.
	
	\medskip
	\noindent\textit{Data Availability Statement.} Data sharing is not applicable to this article as no datasets were generated or analysed during the current study.
	
	\medskip
	\noindent\textit{Acknowledgments.} The authors would like to sincerely thank the referee for several useful comments improving the paper.
	
	\medskip


\begin{thebibliography}{99}
		
		\bibitem{AA} 
		I. Amemiya and T. Ando, \textit{Convergence of random products of contractions in Hilbert space}, 
		Acta Sci. Math. (Szeged) \textbf{26} (1965), 239--244.
		
		\bibitem{BS} C. Badea and D. Seifert, \textit{Ritt operators and convergence in the method of alternating projections}, J. Approx. Theory \textbf{205} (2016), 133--148.
		
		\bibitem{BOR1} P. Borodin and E. Kopeck\'{a}, \textit{Consecutive projections and greedy approximation in Hilbert space}, J. Math. Sci. (N.Y.) \textbf{299} (2026), no. 3, 314--343.
		
		\bibitem{BOR2} P. Borodin and E. Kopeck\'{a}, \textit{Convergence of remote projections onto convex sets}, Pure Appl. Funct. Anal. \textbf{8} (2023), no. 6, 1603--1620.
		
		\bibitem{BOR3} P. Borodin and E. Kopeck\'{a}, \textit{Weak limits of consecutive projections and of greedy steps}, Proc. Steklov Inst. Math. \textbf{319} (2022), no. 1, 56--63.
		
		\bibitem{DEU2} F. Deutsch, \textit{The method of alternating orthogonal projections}, 
		In: Singh, S.P. (ed.) Approximation Theory, Spline Functions and Applications, 
		NATO ASI Series, vol. 356, Springer, Dordrecht, 1992.
		
		\bibitem{DYE} J. M. Dye and S. Reich, \textit{On the unrestricted iteration of projections in Hilbert space}, J. Math. Anal. Appl. \textbf{156} (1991), no. 1, 101--119.
		
		\bibitem{DYE2} J. M. Dye and S. Reich, \textit{Unrestricted iterations of nonexpansive mappings in Hilbert space}, Nonlinear Anal. \textbf{18} (1992), no. 2, 199--207.
		
		\bibitem{EM} R. Eskandari and M. S. Moslehian, \textit{Convergence of random products of countably infinitely many projections}, J. Approx. Theory \textbf{317} (2026), 106286.
		
		\bibitem{EBM} R. Eskandari, A. M. Bikchentaev, and M. S. Moslehian, \textit{Generalizations of von Neumann's alternating projections theorem to contractions and positive operators}, 
		Linear Alg. Appl., to appear, https://doi.org/10.1016/j.laa.2026.03.022
		
		\bibitem{N} I. Halperin, \textit{The product of projection operators}, Acta Sci. Math. (Szeged) \textbf{23} (1962), no. 2, 96--99.
		
		\bibitem{KM} E. Kopeck\'{a} and V. M\"{u}ller, \textit{A product of three projections}, 
		Studia Math. \textbf{223} (2014), no. 2, 175--186.
		
		\bibitem{KP} E. Kopeck\'{a} and A. Paszkiewicz, \textit{Strange products of projections}, Israel J. Math. \textbf{219} (2017), no. 1, 271--286. 
		
		\bibitem{KR1} E. Kopeck\'{a} and S. Reich, \textit{A note on the von Neumann alternating projections algorithm}, J. Nonlinear Convex Anal. \textbf{5} (2004), no. 3, 379--386.
		
		\bibitem{KR2} E. Kopeck\'{a} and S. Reich, \textit{Another note on the von Neumann alternating projections algorithm}, J. Nonlinear Convex Anal. \textbf{11} (2010), no. 3, 455--460.
		
		\bibitem{MDD} \'{I}. D. L. Melo, J. X. da Cruz Neto, and J. M. Machado de Brito, \textit{Strong convergence of alternating projections}, J. Optim. Theory Appl. \textbf{194} (2022), no. 1, 306--324.
		
		\bibitem{Mur} G. J. Murphy, \textit{$C^*$-Algebras and Operator Theory}, Academic Press, New York, 1990.
		
		\bibitem{PAS} A. Paszkiewicz, \textit{The Amemiya--Ando conjecture falls}, arXiv:1203.3354.
		
		\bibitem{PRZ} E. Pustylnik, S. Reich, and A. J. Zaslavski, \textit{Inner inclination of subspaces and infinite products of orthogonal projections}, J. Nonlinear Convex Anal. \textbf{14} (2013), no. 3, 423--436.
		
		\bibitem{PUS1} E. Pustylnik, S. Reich, and A. J. Zaslavski, \textit{Convergence of non-periodic infinite products of orthogonal projections and nonexpansive operators in Hilbert space}, 
		J. Approx. Theory \textbf{164} (2012), no. 5, 611--624.
		
		\bibitem{PUS2} E. Pustylnik, S. Reich, and A. J. Zaslavski, \textit{Convergence of non-cyclic infinite products of operators}, J. Math. Anal. Appl. \textbf{380} (2011), no. 2, 759--767.
		
		\bibitem{REI} S. Reich, \textit{The alternating algorithm of von Neumann in the Hilbert ball}, Dynam. Systems Appl. \textbf{2} (1993), no. 1, 21--25. 
		
		\bibitem{RZ1} S. Reich and R. Zalas, \textit{Polynomial estimated for the method of cyclic projections in Hilbert spaces}, Numerical Algorithms \textbf{94} (2023), 1217--1242.
		
		\bibitem{RZ2} S. Reich and R. Zalas, \textit{Comparing the methods of alternating and simultaneous projections for two subspaces}, Linear Algebra Appl. 683 (2024), 235--263.
				
		\bibitem{SAK} M. Sakai, \textit{Strong convergence of infinite products of orthogonal projections in Hilbert space}, Appl. Anal. \textbf{59} (1995), no. 1–4, 109--120.
		
		\bibitem{THI} D. K. Thimm, \textit{Most iterations of projections converge}, J. Optim. Theory Appl. \textbf{203} (2024), no. 1, 285--304.
		
		\bibitem{THI2} D. K. Thimm, \textit{On a meager full measure subset of $N$-ary sequences}, Appl. Set-Valued Anal. Optim. \textbf{6} (2024), 81--86.
		
		\bibitem{VON} J. von Neumann, \textit{On rings of operators. Reduction theory}, Ann. of Math. \textbf{50} (1949), no. 2, 401--485.
		
	\end{thebibliography}
\end{document}